\documentclass[12pt,oneside, a4paper]{article}
\pdfoutput=1
\usepackage{a4} 
\usepackage{times} 
\usepackage[affil-it]{authblk} 
\usepackage{amsthm,amsmath,amssymb}
\usepackage{thmtools} 
\usepackage{mathrsfs}
\usepackage{amsfonts}
\usepackage{mathtools}
\usepackage{algorithm}
\usepackage{algorithmic}
\usepackage{bbm}
\usepackage{refcount}
\usepackage{color}
\usepackage{textgreek}
\usepackage[normalem]{ulem} 
\usepackage{soul} 

\usepackage{fancyhdr}

\newtheorem{Theorem}{Theorem}
\newtheorem{Lemma}{Lemma}

\newtheorem{Remark}{Remark}
\newtheorem*{Asymptotics}{Asymptotics}
\newtheorem{Example}{Example}

\theoremstyle{definition}
\newtheorem{Condition}{Condition}

\DeclareMathOperator*{\argmin}{arg\,min}
\DeclareMathOperator*{\argmax}{arg\,max}

\title{Concentration behavior of the penalized least squares estimator 
~\\
~\\
 \fontsize{11}{13}\selectfont Penalized least squares behavior}

\author{%
Alan Muro and Sara van de Geer\\ 
\texttt{ \fontsize{11}{13}\selectfont \{muro,geer\}@stat.math.ethz.ch}}

\affil{\fontsize{11}{13}\selectfont Seminar f\"ur Statistik, ETH Z\"urich\\ R\"amistrasse 101, 8092 Z\"urich \\  }

\date{}

\usepackage[style=authoryear,bibstyle=authoryear,doi=false,url=false,
eprint=false,isbn=false,firstinits=true,babel=hyphen,
maxcitenames=2,maxbibnames=5, dashed=false,backend=bibtex]{biblatex}

\renewbibmacro{in:}{}
\DeclareFieldFormat[article]{title}{#1}
\DeclareFieldFormat{pages}{#1} 
\begin{document}
\maketitle

\begin{abstract}
Consider the standard nonparametric regression model and take as estimator the penalized least squares function. In this article, we study the trade-off between closeness to the true function and complexity penalization of the estimator, where complexity is described by a seminorm on a class of functions. First, we present an exponential concentration inequality revealing the concentration behavior of the trade-off of the penalized least squares estimator around a nonrandom quantity, where such quantity depends on the problem under consideration. 
Then, under
some conditions and for the proper choice of the tuning parameter, we obtain bounds for this nonrandom quantity. 
We illustrate our results with some examples that include the smoothing splines estimator.statistical trade-off. 
\end{abstract}

\section{Introduction}\label{S:Intro}

Let $Y_1,...,Y_n$ be independent real-valued response variables satisfying 

\begin{equation*}
Y_i = f^0(x_i) + \epsilon_i, \;i=1,...,n, 
\end{equation*}

where $x_1,...,x_n$ are given covariates in some space $\mathcal{X}$, $f^0$ is an unknown function in a given space $\mathcal{F}$, and $\epsilon_1,...,\epsilon_n$ are
independent standard Gaussian random variables. We assume that $f^0$ is
``smooth'' in some sense but that this degree of smoothness is unknown. We
will make this clear below.\\
~\\
To estimate $f^0$, we consider the penalized least squares estimator given by

\begin{equation}\label{eq:fHat}
\hat{f} := \argmin_{f\in\mathcal{F}} \left\lbrace ||Y - f||_n^2 + \lambda^2
  \mathcal{I}^2 (f) \right\rbrace,  
\end{equation}

where $\lambda>0$ is a tuning parameter and $\mathcal{I}$ a given seminorm on $\mathcal{F}$. Here, for a vector $u \in \mathbb{R}^n$, we write $||u||_n^2:= u^T u /n$, and we apply the same notation $f$ for the vector $f = (f(x_1),...,f(x_n))^T$ and the function $f\in\mathcal{F}$. Moreover, to avoid digressions from our main arguments, we assume that expression \eqref{eq:fHat} exists and is unique. \\
~\\
For a function $f\in\mathcal{F}$, define

\begin{equation}\label{tau}
\tau^2(f) :=  ||f - f^0||^2_n + \lambda^2
\mathcal{I}^2(f). 
\end{equation}  
  

This expression can be seen as a description of the trade-off of $f$. The term $||f-f^0||_n$ measures the closeness of $f$ to the true function, while $\mathcal{I}(f)$ quantifies its ``smoothness''. As mentioned above, we only assume that $f^0$ is not too complex, but that this degree of complexity is not known. This accounts for assuming that $\mathcal{I}(f^0) < \infty$, but that an upper bound for $\mathcal{I}(f^0)$ is unknown. Therefore, we choose as model class $\mathcal{F} = \{ f \, : \, \mathcal{I}(f) < \infty\}$. It is important to see that taking as model class $\mathcal{F}_0=\{ f : \mathcal{I}(f) \leq M_0 \}$, for some fixed $M_0 >0$, instead of $\mathcal{F}$ could lead to a model misspecification error if the unknown smoothness of $f^0$ is large. \\
~\\
Estimators with roughness penalization have been widely studied. \textcite{wahba1990spline} and \textcite{green1993nonparametric} consider the smoothing splines estimator, which corresponds to the solution
of \eqref{eq:fHat} when  $\mathcal{I}^2(f) = \int_0^1 |f^{(m)}(x)|^2
  \mathrm{d}x,$ where $f^{(m)}$  denotes the $m$-th derivative of $f$. \textcite{gu2002smooth} provides results for the more general penalized likelihood estimator with a general quadratic functional as complexity regularization. If we assume that this functional is a seminorm and that the noise follows a Gaussian distribution, then the penalized likelihood estimator reduces to the
estimator in \eqref{eq:fHat}.\\
~\\
Upper bounds for the estimation error can be found in the literature (see e.g. \textcite{Vaart}, \textcite{barrio2007lectures}). When no complexity regularization term is included, the standard method used to derive these is roughly as follows: first, the basic inequality 

\begin{equation}\label{eq:introBasicIneq}
||\hat{f}-f^0||^2_n \leq \frac{2}{n} \sup_f \left( \sum_{i=1}^n
\epsilon_i( f(x_i)-f^0(x_i)) \right)
\end{equation}

 is invoked.  Then, an inequality for the right hand side of \eqref{eq:introBasicIneq} is obtained for functions $f$ in $\{g\in\mathcal{F}: ||g-f^0||_n \leq R \}$ for some $R >0$. Finally, upper bounds for the estimation error are obtained with high probability by using entropy computations. When a penalty term is included, a similar approach can be used, but in this case the process in \eqref{eq:introBasicIneq} is
studied in terms of both $||f-f^0||_n$ and the smoothness of the functions
considered, i.e., $\mathcal{I}(f)$ and $\mathcal{I}(f^0)$.\\
~\\
A limitation of the approach mentioned above is that it does not allow us to obtain lower bounds. Consistency results have been proved (e.g. \textcite{vdg96}), but it is not clear how to use
these to derive explicit bounds. \textcite{ChatLSE2014} proposes a new approach to estimate the exact value of the error of least square estimators under convex constraints. It provides a concentration result for the estimation error by relating it with the expected maxima of a Gaussian process. One can then use this result to get both upper and lower bounds. In \textcite{vdGW2016}, a more ``direct" argument is employed to show that the error for a more general class of penalized least squares estimators is concentrated around its expectation. Here, the penalty is only assumed to be convex. Moreover, the authors also consider the approach from \textcite{ChatLSE2014} to derive a concentration result for the trade-off $\tau(\hat{f})$ for uniformly bounded function classes under general loss functions.\\
~\\
The goal of this paper is to contribute to the study of the ``behavior'' of $\tau(\hat{f})$ from a theoretical point of view. We consider the approach from \textcite{ChatLSE2014} and extend the ideas to the penalized least squares estimator with penalty based on a squared seminorm without making assumptions on the function space. We present a concentration inequality showing that $\tau(\hat{f})$ is concentrated around a nonrandom quantity $R_0$ (defined below) in the nonparametric regime. Here, $R_0$ depends on the sample size and the problem under consideration. Then, we derive upper and lower bounds for $R_0$ in Theorem \ref{Theorem:BoundsTau}. These are obtained for the proper choice of $\lambda$ and two additional conditions, including an entropy assumption. Combining this result with Theorem \ref{Theorem_CI}, one can obtain both upper and lower bounds for $\tau(\hat{f})$ with high probability for a sufficiently large sample size. 
We illustrate our results with some examples in section \ref{ss:Examples} and observe that we are able to recover  optimal rates of convergence for the estimation error from the literature.\\
~\\
We now introduce further notation that will be used in the following sections. Denote the minimum of $\tau$ over $\mathcal{F}$ by 
\begin{equation*}
R_{min} := \min_{f\in\mathcal{F}} \tau (f),
\end{equation*}

and let this minimum be attained by $f_{min}$. Note that $f_{min}$ can be seen as the unknown noiseless counterpart of
$\hat{f}$ and that $R_{min}$ is
a nonrandom unknown quantity. For two vectors $u,v\in\mathbb{R}^n$, let $\langle u , v\rangle$ denote the usual inner product. For $R \geq R_{min}$ and $\lambda >0$, define 

\begin{equation*}
 \mathcal{F}(R) := \left\lbrace f\in\mathcal{F} : \tau(f) \leq R \right\rbrace,
\end{equation*}

and

\begin{equation*}
M_n(R): = \sup_{f\in\mathcal{F}(R)}  \langle \epsilon, f- f^0 \rangle/n. 
\end{equation*}

Additionally, we write 

\begin{equation*}
M(R) := \mathbb{E} M_n(R),\quad\quad H_n(R): = M_n(R) - \frac{R^2}{2},\quad\quad H(R): = \mathbb{E} H_n(R).
\end{equation*}
  

Moreover, define the random quantity
\begin{equation*}
 R_* := \argmax_{R \geq R_{min}} H_n(R),
\end{equation*}   
and the nonrandom quantity
\begin{equation*}
R_0 := \argmax_{R \geq R_{min}} H(R).
\end{equation*} \\
~\\
From Lemma \ref{Lemma:H_concave}, it will follow that $R_{*}$ and $R_0$ are unique.\\
~\\
For ease of exposition, we will use the following asymptotic notation throughout this paper: for two positive sequences $\{x_n\}_{n=1}^\infty$ and
 $\{y_n\}_{n=1}^\infty$, we write $x_n=\mathcal{O}(y_n)$ if $\limsup |x_n/y_n| < \infty$ as $n \to \infty$ and $x_n=o(y_n)$ if $x_n/y_n \to 0$ as $n \to \infty$. Moreover, we employ the notation $x_n \asymp y_n$ if $x_n = \mathcal{O}(y_n)$ and $y_n = \mathcal{O}(x_n)$. In addition, we make use of the stochastic order symbols $\mathcal{O}_{P}$ and $o_{P}$.\\
~\\
It is important to note that the quantities $\lambda$, $\tau(\hat{f})$, $\tau(f^0)$, $R_{min}$, $M(R)$, $H(R)$,  $R_{*}$, and $R_0$ depend on the sample size $n$. However, we omit this dependence in the notation to simplify the exposition.

\subsection{Organization of the paper}\label{Ss:OrganizationPaper}

First, in section \ref{Ss:condANDresults}, we present the main results: Theorems \ref{Theorem_CI} and  \ref{Theorem:BoundsTau}. Note that the former does not require further assumptions than those from section \ref{S:Intro}, while the latter needs two extra conditions, which will be introduced after stating Theorem \ref{Theorem_CI}. Then, in section \ref{ss:Examples}, we illustrate the theory with some examples and, in section \ref{ss:Conclusions}, we present some concluding remarks. Finally, in section  \ref{S:Proofs}, we present all the proofs. We deferred to the appendix results from the literature used in this last section.

\section{Behavior of $\tau(\hat{f})$}\label{S:Result}

\subsection{Main results}\label{Ss:condANDresults}

The first theorem provides a concentration probability inequality for
$\tau(\hat{f})$ around $R_0$. It can be seen as an extension
of Theorem 1.1 from \textcite{ChatLSE2014}
to the penalized least squares estimator with a squared seminorm on $\mathcal{F}$ in the penalty term.

\begin{Theorem}\label{Theorem_CI}
For all $\lambda>0$ and $x>0$, we have
\begin{equation*}
\mathbb{P} \left( \Bigg\vert  \frac{\tau(\hat{f})}{R_0} - 1\Bigg\vert \geq x  \right) \leq 3 \exp \left(- \frac{x^4 (\sqrt{n}R_0)^2}{32 (1 + x)^2} \right).
\end{equation*}
\end{Theorem}
~\\
\begin{Asymptotics} \label{Asymptotics_Thm1}
If $R_0$ satisfies $1/(\sqrt{n} R_0) = o(1)$, then

\begin{equation*}
\Bigg\vert \frac{\tau(\hat{f})}{R_0} - 1  \Bigg\vert = o_\mathbb{P}(1).
\end{equation*}

Therefore, the random fluctuations of $\tau(\hat{f})$ around $R_0$ are of negligible size in comparison with $R_0$. Moreover, this asymptotic result  implies that the asymptotic distribution of $\tau(\hat{f})/ R_0$ is degenerate. In Theorem \ref{Theorem:BoundsTau}, we will provide bounds for $R_0$ under some additional conditions and observe that $R_0$ satisfies the condition from above. 
\end{Asymptotics}

\begin{Remark}
Note that we only consider a square seminorm in the penalty term. A generalization of our results to penalties of the form $\mathcal{I}^q(\cdot),\, q \geq 2$, is straightforward but omitted for simplicity. The case $q<2$ is not considered in our study since our method of proof requires that the square root of the penalty term is convex, as can be observed in the proof of Lemma \ref{Lemma:H_concave}. 
\end{Remark}

Before stating the first condition required in Theorem \ref{Theorem:BoundsTau}, we will introduce the following definition:
let $S$ be some subset of a metric space $(\mathcal{S}, d)$. For $\delta
>0$, the \textit{$\delta$-covering number} $N(\delta , S, d)$ of $S$ is the smallest
value of $N$ such that there exist $s_1,...,s_N$ in $\mathcal{S}$ such that

\begin{equation*}
\min_{j=1,...,N} d(s, s_j) \leq \delta, \,\, \forall \, s\in S.
\end{equation*}
~\\
Moreover, for $\delta>$0, the \textit{$\delta$-packing number} $N_{D} (\delta, S, d)$ of $S$
is the largest value of N such that there exist $s_1,...,s_N$ in
$\mathcal{S}$ with
\begin{equation*}
d(s_k,s_j) > \delta, \,\, \forall k \neq j.
\end{equation*}

\begin{Condition}\label{Ss:CoveringNumberCondition} Let $ \alpha \in(0,2)$.
For all $\lambda > 0$ and $R \geq R_{min}$, we have
%

\begin{equation*}
 \log N(u, \mathcal{F}(R),||\cdot ||_n) = \mathcal{O}\left( \left(\frac{R/\lambda}{u}\right)^\alpha \right),\,\, u >0,
\end{equation*}
and for some $c_0 \in (0,2]$,
\begin{equation*}
\frac{1}{\log  N(c_0R, \mathcal{F}(R),||\cdot ||_n)} = \mathcal{O} \left( (c_0 \lambda) ^{\alpha}\right).
\end{equation*}

~\\
Condition \ref{Ss:CoveringNumberCondition} can be seen as a description of
the richness of $\mathcal{F}(R)$. We refer the reader to \textcite{SelectedKolmogorov} and \textcite{BirSolo1967} for an extensive study on entropy bounds, and to \textcite{Vaart} and \textcite{vdG2009empirical} for their application in empirical process theory.~\\

\end{Condition}

%

\begin{Condition}\label{Ss:RminCondition}  
$R_{min} \asymp \tau(f^0)$.\\

Condition \ref{Ss:RminCondition} relates the roughness penalty
of $f^0$ with the minimum trade-off achieved by functions in $\mathcal{F}$. It implies that our choice of the penalty term is appropriate. In other words, $\mathcal{I}(f_{min})$ and $\mathcal{I}(f^0)$ are not ``too far away" from each other when the tuning parameter is chosen properly. Therefore, aiming to mimic the trade-off of $f_{min}$ points us in the right direction as we would like to estimate $f^0$.
 \end{Condition}

The following result provides upper and lower bounds
for $R_0$. Note that $\mathcal{I}(f^0)$ is permitted to depend on the sample size. We assume that $\mathcal{I}(f^0)$ remains upper bounded and bounded away from zero as the sample size increases. However, we allow these bounds to be unknown. 

\begin{Theorem}\label{Theorem:BoundsTau}
Assume Conditions \ref{Ss:CoveringNumberCondition} and
\ref{Ss:RminCondition}, and that $\mathcal{I}(f^0) \asymp 1$. For

\begin{equation*}
\lambda \asymp n^{-\frac{1}{2+\alpha}}
\end{equation*} 

suitably chosen depending on $\alpha$, one has 

\begin{equation*}
R_0 \asymp n^{-\frac{1}{2+\alpha}}.
\end{equation*}

Therefore, we obtain

\begin{equation*}
\tau(\hat{f}) \asymp n^{-\frac{1}{2+\alpha}}
\end{equation*}

with probability at least $1- 3 \exp\left( -c' n^{\frac{\alpha}{2 + \alpha}} \right) - 3 \exp\left( -c'' n^{\frac{\alpha}{2 + \alpha}} \right)$,
where $c', c''$ denote some positive constants not depending on the sample size.

\end{Theorem}
~\\
Theorem \ref{Theorem:BoundsTau} shows that one can obtain bounds for $R_0$ if we choose the tuning parameter properly. Since the condition $1/(\sqrt{n} R_0) = o(1)$ is satisfied, one obtains then convergence in probability of the ratio $\tau(\hat{f})/R_0$ to a constant. Moreover, the theorem also provides bounds for the trade-off of $\hat{f}$. Note that these hold with probability tending to one.

\begin{Remark}
From Theorem \ref{Theorem:BoundsTau}, one obtains 

\begin{equation*}
 ||\hat{f} - f^0||^2_n + \lambda^2 \mathcal{I}^2(\hat{f}) \asymp \min_{f\in\mathcal{F}}
\Big( ||f- f^0||^2_n + \lambda^2 \mathcal{I}^2(f) \Big)
\end{equation*}

with high probability for a sufficiently large sample size. One can then say that the trade-off of $\hat{f}$ mimics that of its noiseless counterpart $f_{min}$, i.e., $\hat{f}$ behaves as if there were no noise in our observations. Therefore, our choice of $\lambda$ allows us to control the random part of the problem using the penalty term and ``over-rule'' the variance of the noise. 

\end{Remark}

\begin{Remark}\label{remark:RecoverRates}
From Theorem \ref{Theorem:BoundsTau}, one can observe that

\begin{equation*}
|| \hat{f}- f^0 ||_n = \mathcal{O}_{P}(n^{-\frac{1}{2+\alpha}}),
\end{equation*}

\begin{equation*}
\mathcal{I}(\hat{f})= \mathcal{O}_{P}(1).
\end{equation*}


Therefore, we are able to recover the rate of convergence for the estimation error of
$\hat{f}$. This will be illustrated in section \ref{ss:Examples}. Furthermore, we observe that the degree of smoothness of $\hat{f}$ is bounded in probability. Although we have a lower bound for $\tau(\hat{f})$, this does neither imply (directly) a lower bound for $||\hat{f}-f^0||_n$, nor for $\mathcal{I}(\hat{f})$.
\end{Remark}

\subsection{Examples}\label{ss:Examples} 

In the following, we require that the assumptions in Theorem \ref{Theorem:BoundsTau} hold. In each example, we provide references to results from the literature where one can verify that Condition \ref{Ss:CoveringNumberCondition} is satisfied and refer the interested reader to these for further details. For the lower bounds, one may first note that
\begin{align*}
~
&N(u, \mathcal{F}(R), || \cdot ||_n) \\
& \geq N \left(u, \left\lbrace f\in\mathcal{F} : ||f-f^0||_n \leq \frac{R} {\sqrt{2}},\, \mathcal{I}(f) \leq \frac{R}{ \sqrt{2} \lambda}  \right\rbrace, || \cdot ||_n \right)
\end{align*}

and additionally insert the results from \textcite{yang1999}, where the authors show that often global and local entropies are of the same order for some $f^0$.
%
%

\begin{Example}\label{Example:SS}
Let $\mathcal{X}=[0,1]$ and $\mathcal{I}^2(f) = \int_0^1 |f^{(m)}(x)|^2
\mathrm{d}x$ for some $m\in\{2,3,...\}$. Then $\hat{f}$ is the smoothing spline estimator and can be explicitly
computed (e.g. \textcite{green1993nonparametric}). Moreover, it can be shown that in this case Condition \ref{Ss:CoveringNumberCondition} holds with $\alpha = 1/m$ under some conditions on the design matrix (\textcite{SelectedKolmogorov} and Example 2.1 in \textcite{vdg1990}). Therefore from Theorem \ref{Theorem:BoundsTau}, we have that the standard choice $\lambda \asymp n^{-\frac{m}{2m+1}}$ yields 
\begin{equation*}
R_0 \asymp
n^{-\frac{m}{2m+1}}.
\end{equation*}

Moreover, we obtain upper and lower bounds for $\tau(\hat{f})$  with high probability for large $n$. Then by Remark \ref{remark:RecoverRates}, we recover the optimal rate of convergence for $||\hat{f}-f^0||_n$ (\textcite{stone1982}). \\
  ~\\
Now we consider the case where the design points have a larger dimension. Let $\mathcal{X}=[0,1]^d$ with $d \geq 2$ and define the $d$-dimensional index $r=(r_1,...,r_d)$, where the values $r_i$ are non-negative integers. We write $|r|=\sum_{i=1}^d r_i$. Furthermore, denote by $D^{r}$ the differential operator defined by

\begin{equation*}
D^r f(x)= \frac{\partial ^{|r|}}{\partial x_1^{r_1} \cdot\cdot\cdot \partial x_d^{r_d}} f(x_1,...,x_d)
\end{equation*}

and consider as roughness penalization 

\begin{equation*}
\mathcal{I}^2(f) = \sum_{|r|=m} \int | D^r f(x)  |^2 \mathrm{d}x,
\end{equation*}

with $m > d/2$.  In this case, Condition \ref{Ss:CoveringNumberCondition} holds with $\alpha = d/m$ (\textcite{BirSolo1967}) 
 Then by Theorem \ref{Theorem:BoundsTau}, we have that for the choice $\lambda \asymp n^{-\frac{m}{2m + d}}$  we obtain

\begin{equation*}
R_0 \asymp n^{-\frac{m}{2m + d}}.
\end{equation*}

Similarly as above, we are able to recover the optimal rate for the estimation error. 
\end{Example}

\begin{Example}\label{Example:TV}
In this example, define the total variation penalty as
\begin{equation*}
TV(f) = \sum_{i=2}^n | f(x_i) - f(x_{i-1}) |,
\end{equation*}
where $x_1,...,x_n$ denote the design points.\\
~\\
Let $\mathcal{X}$ be real-valued and $\mathcal{I}^2(f) =(TV(f))^2$. In this case, Condition \ref{Ss:CoveringNumberCondition} is fulfilled for $\alpha=1$ (\textcite{BirSolo1967}). 
The advantage of the total variation penalty over that from Example
\ref{Example:SS} is that it can be used for unbounded $\mathcal{X}$. \\
~\\
Let $\alpha=1$ in Theorem \ref{Theorem:BoundsTau}. For the choice $\lambda \asymp n^{-\frac{1}{3}}$, we have 

\begin{equation*}
R_0 \asymp n^{-\frac{1}{3}},
\end{equation*}

and we also obtain bounds for $\tau(\hat{f})$ with high probability for large $n$. By Remark \ref{remark:RecoverRates}, we also recover the optimal upper bound for the estimation error. 
\end{Example}

\subsection{Conclusions}\label{ss:Conclusions}

Theorem \ref{Theorem_CI} derives a concentration result for $\tau(\hat{f})$ around a nonrandom quantity rather than just an upper bound. In particular, we observe that the ratio $\tau(\hat{f})/R_0$ convergences in probability to $1$ if  $R_0$ satisfies $1 / (\sqrt{n}R_0) = o(1)$. This condition holds in the nonparametric setting for $\lambda$ suitably chosen and $\mathcal{I}(f^0) \asymp 1$, as shown in Theorem \ref{Theorem:BoundsTau} and in our examples from section \ref{ss:Examples}.\\
~\\
The strict concavity of $H$ and $H_n$ (Lemma \ref{Lemma:H_concave}) plays an important role in the derivation of both theorems. In our work, the proof of this property requires that the square root of the penalty term is convex. 
Furthermore, the proof of both Theorems \ref{Theorem_CI} and \ref{Theorem:BoundsTau} rely on the fact that the noise vector is Gaussian. This can be seen in Lemma \ref{ConcIneqG}, where we invoke a concentration result for functions of independent Gaussian random variables, and in Lemma \ref{Lemma_BoundsHR}, where we employ a lower bound for the expected value of the supremum of Gaussian processes to bound the function $H$. 

\section{Proofs}\label{S:Proofs}

This section is divided in two parts. In section \ref{Ss:Proofs_Lemmas}, we first state and prove the lemmas necessary to
prove Theorem \ref{Theorem_CI}. These follow closely the proof of Theorem 1.1
from \textcite{ChatLSE2014}, however, here we include a roughness penalization term for
functions in $\mathcal{F}$. At the end of this section, we combine these lemmas to prove the first theorem. In section \ref{Ss:Proofs_Theorems}, we first prove an additional result necessary to establish Theorem \ref{Theorem:BoundsTau}. After this, we present the proof of the second theorem. Some results from the literature used in our proofs are deferred to the appendix. 

\subsection{Proof of Theorem \ref{Theorem_CI}}\label{Ss:Proofs_Lemmas}

\begin{Lemma}\label{Lemma:H_concave}
For all $\lambda >0$, $H_n(\cdot)$ and $H(\cdot)$ are strictly concave functions.
\end{Lemma}

\begin{proof}

Let $\lambda >0$. Take any two values $r_s,r_b$ such that $R_{min} \leq r_s \leq
r_b$ and define 

\begin{equation*}
\mathcal{F}_{s,b} := \{ f_t=tv_s+(1-t)v_b \, | \, t\in [0,1],
v_s\in\mathcal{F}(r_s), v_b\in\mathcal{F}(r_b) \}. 
\end{equation*}

Take $v_t\in
\mathcal{F}_{s,b}$ and let $r =tr_s + (1-t)r_b$. By properties of a
seminorm, we have that $\mathcal{I}(v_t) < \infty$, which implies that $v_t \in
\mathcal{F}$. Moreover, we have that

\begin{align*}
\tau(v_t) 
& = \sqrt{ ||v_t - f^0||_n^2 + \lambda^2 \mathcal{I}^2(v_t) }\\
& \leq t \sqrt{ ||v_s - f^0||_n^2 + \lambda^2 \mathcal{I}^2(v_s)  } + (1-t) \sqrt{ ||v_b - f^0||_n^2 + \lambda^2 \mathcal{I}^2(v_b)  }\\
& \leq t \tau(v_s) + (1-t) \tau(v_b)
\leq t r_s + (1-t) r_b,
\end{align*}

where the first inequality uses the fact that the square root of the penalty term $\sqrt{\lambda^2 \mathcal{I}^2(\cdot)}$ is convex. Therefore, $v_t\in\mathcal{F}(r)$. Using these equations, we have

\begin{align}\label{proof:aux_ineq_Mn}
M_n(r) = \sup_{f\in\mathcal{F}(r)}  \langle \epsilon, f- f^0 \rangle/n
& \geq  \sup_{f\in\mathcal{F}_{s,b}}  \langle \epsilon, f- f^0 \rangle/n\\
& = \sup_{\substack{v_s,v_b\in\mathcal{F} \\ \tau(v_s) \leq r_s \\
    \tau(v_b)\leq r_b}}  \langle \epsilon, tv_s + (1-t)v_b - f^0
\rangle/n \nonumber\\
& = t M_n(r_s) + (1-t) M_n(r_b). \nonumber
\end{align}

Therefore, $M_n$ is concave for all $\lambda >0$. Taking expected value in the equations in \eqref{proof:aux_ineq_Mn} yields that $M$ is concave for all $\lambda >0$. Since $g(r) := - \frac{r^2}{2}$ is strictly concave, then $H_n$ and $H$ are strictly concaves and we have our result.
\end{proof}

%
%
%
%

\begin{Lemma}\label{Lemma:UniquenessLSE}
For all $\lambda >0$, we have that $\tau(\hat{f}) = R_*$. 
\end{Lemma}

\begin{proof}
Let $f^* \in \mathcal{F}(R_*)$ be such that

\begin{equation*}
\langle \epsilon , f^* - f^0 \rangle /n = \sup_{f\in\mathcal{F}(R_*)}  \langle \epsilon, f- f^0 \rangle/n.
\end{equation*}
We will show first that $\tau(f^*) = R_*$. Suppose $\tau(f^*) = \tilde{R}$ for some $R_{min} \leq \tilde{R} < R_*$. Note that then $M_n(\tilde{R}) = M_n(R_*)$, and therefore, we have

\begin{equation*}
H_n(\tilde{R}) = H_n(R_*) + \left(
  \frac{R_*}{2} - \frac{\tilde{R}}{2}  \right) > H_n(R_*),
\end{equation*}

which is a contradiction by definition of $R_*$. We must then have that $\tau(f^*) = R_*$. \\
~\\
Now we will prove that $\tau(\hat{f}) = \tau(f^*)$. For all $\lambda >0$ and for all $f\in\mathcal{F}$, we have 

\begin{align}\label{eq:aux_proof_uniqueness}
~ 
& || Y - f ||^2_n + \lambda^2 \mathcal{I}^2(f)\\ \nonumber
& = || Y - f^0 ||^2_n - 2 \left\lbrace \langle \epsilon, f - f^0 \rangle / n - \frac{1}{2} \Bigg( || f - f^0||_n^2 + \lambda^2 \mathcal{I}^2(f) \Bigg) \right\rbrace\\ \nonumber
& \geq || Y - f^0 ||^2_n - 2 H_n(\sqrt{|| f - f^0||_n^2 + \lambda^2 \mathcal{I}^2(f)})\\ \nonumber
& \geq || Y - f^* ||^2_n + \lambda^2 \mathcal{I}^2(f^*).\nonumber
\end{align}

In consequence, by definition of $\hat{f}$ and the inequalities in \eqref{eq:aux_proof_uniqueness}, both $\hat{f}$ and $f^*$ minimize 
\begin{equation*}
|| Y - f ||^2_n + \lambda^2 \mathcal{I}^2(f). 
\end{equation*}

and by uniqueness, it follows that $\tau(\hat{f})=\tau(f^*)$.
\end{proof}

%
%
%
%
%
%
%
%

\begin{Lemma}\label{Lemma:control_E}
For $x >0$, define

\begin{equation*}
s_1 :=R_0 - x,\,\, s_2 :=R_0 +x,\,\, z := H(R_0) - x^2/4.
\end{equation*} 
Moreover, define the event $e = \left\lbrace \left\lbrace H_n(s_1) < z
  \right\rbrace\;\wedge\;\left\lbrace H_n(s_2) < z\right\rbrace\wedge
  \left\lbrace H_n(R_0) > z \right\rbrace \right\rbrace$.\\ We have
\begin{equation*}
\mathbb{P}\left( e^c \right) \leq 3 \exp \left(-\frac{n\,x^4}{32(R_0 + x)^2} \right).
\end{equation*}

\end{Lemma}

\begin{proof}

From the proof of Theorem 1.1 in \textcite{ChatLSE2014}, one can easily observe that for all $\lambda>0$ and any $R$ we have 
\begin{equation*}
H(R_0) - H(R) \geq  \frac{(R - R_0)^2}{2}.
\end{equation*}

Applying the inequality above to $R=s_1$ and $R=s_2$, we have that
 
\begin{equation*}
H(s_i) + \frac{x^2}{4} \leq
H(R_0) - \frac{x^2}{4}, \quad\quad i=1,2.
\end{equation*}

By Lemma \ref{ConcIneqG}, we have

\begin{equation*}
\mathbb{P}\left( H_n(R_0) \leq z \right) = \mathbb{P}\left( H_n(R_0) \leq
  H(R_0) -\frac{x^2}{4} \right) \leq e^{-\frac{n\,x^4}{32(R_0)^2}}. 
\end{equation*}

Therefore, using again Lemma \ref{ConcIneqG} yields

\begin{equation*}
\mathbb{P}\left( H_n(s_1) \geq z \right) \leq \mathbb{P}\left( H_n(s_1) \geq H(s_1) + \frac{x^2}{4} \right)
\leq  e^{-\frac{n\,x^4}{32(s_1)^2}} ,
\end{equation*}
and
\begin{equation*}
\mathbb{P}\left( H_n(s_2) \geq z \right) \leq \mathbb{P}\left( H_n(s_2) \geq H(s_2) + \frac{x^2}{4} \right)\leq  e^{-\frac{n\,x^4}{32(s_2)^2}} .
\end{equation*}

By the equations above, we obtain 
\begin{equation*}
\mathbb{P}\left( e^c \right) = \mathbb{P}\left( H_n(s_1)\geq z \right) +
\mathbb{P}\left( H_n(s_2)\geq z \right) + \mathbb{P}\left( H_n(R_0)\leq z
\right) \leq 3e^{-\frac{n\,x^4}{32(s_2)^2}}.
\end{equation*}
\end{proof}

Now we are ready to prove Theorem \ref{Theorem_CI}.

\begin{proof}[Proof of Theorem \ref{Theorem_CI}]
Let $\lambda>0$. First, we note that $H$ is equal to $-\infty$ when $R$ tends to infinity or $R <R_{min}$. Then, by Lemma \ref{Lemma:H_concave}, we know that $R_0$ is unique. \\
~\\
For
$x>0$, define the event
\begin{equation*}
 e := \left\lbrace \left\lbrace H_n(s_1) < z
  \right\rbrace\;\wedge\;\left\lbrace H_n(s_2) < z\right\rbrace\wedge
  \left\lbrace H_n(R_0) > z \right\rbrace \right\rbrace,
\end{equation*}

where $s_1=R_0 - x$, $s_2 =R_0 +x$, and $z = H(R_0) - x^2/4$. Therefore, we have that $s_1 < R_0 < s_2$ by construction. Moreover, we know that $H_n(R_*) \geq H_n(R_0)$ by definition of $R_*$. Since $H_n$ is strictly concave by Lemma \ref{Lemma:H_concave}, we must have that, in $e$,

\begin{equation}\label{proof:Rstar}
s_1 < R_*< s_2.
\end{equation}

Combining Lemma \ref{Lemma:UniquenessLSE} with equation \eqref{proof:Rstar} yields that, in $e$, 
\begin{equation*}
\left\vert \frac{\tau (\hat{f})}
  {R_0} -1  \right\vert < \frac{x}{R_0}. 
\end{equation*}


Therefore, by Lemma \ref{Lemma:control_E}, and letting $y = x / R_0$, we have
 
\begin{equation*}
\mathbb{P} \left( \Bigg\vert  \frac{\tau(\hat{f})}{R_0} - 1\Bigg\vert \geq y  \right) \leq 3 \exp \left(-\frac{n\,y^4 R_0^2}{32(1 + y)^2} \right)
\leq 3 \exp \left(- \frac{y^4 (\sqrt{n}R_0)^2}{32 (1 + y)^2} \right).
\end{equation*}

~\\
%
\end{proof}

\subsection{Proof of Theorem \ref{Theorem:BoundsTau}}\label{Ss:Proofs_Theorems}

For proving Theorem \ref{Theorem:BoundsTau}, we will need the following result. This lemma gives us bounds for the unknown nonrandom quantity $H(R)$. We note that these bounds can be written as parabolas with maximums and maximizers depending on $\alpha$, $n$, and $\lambda$.

\begin{Lemma} \label{Lemma_BoundsHR}		
Assume Condition \ref{Ss:CoveringNumberCondition} and let $\alpha$ be as stated there. For some constants\\ $ C \geq 1/2$, $0 < c_0 \leq 2$, $c_2 > c_1 >0$ not depending on the sample size and for all $\lambda >0$, we have

\begin{equation*}
g_{1}(R) < H(R) < g_{2}(R), \quad R\geq R_{min},
\end{equation*}

where $g_{i}(R) = -\frac{1}{2} (R-K_i)^2 + \frac{K_i^2}{2},\,i=1,2,$
 with

\begin{equation*}
K_1 := K_1(n,\lambda) := \frac{c_0^{1 - \alpha/2}\sqrt{c_1}}{2} \left( \frac{1}{n
    \lambda^\alpha}  \right)^{1/2},
\end{equation*}
\begin{equation*}
 K_2 := K_2(n,\lambda) := \frac{4 C \sqrt{c_2}}{2-\alpha} \left( \frac{1}{n \lambda^\alpha}\right)^{1/2}.
\end{equation*}
\end{Lemma}

\begin{proof}[Proof of Lemma \ref{Lemma_BoundsHR}]
This proof makes use of known results for upper and lower bounds for the expected maxima of random processes, which can be found in the appendix. We will indicate this below.\\
~\\
Let $\lambda>0$ and $R \geq R_{min}$. For $f\in\mathcal{F}(R)$ and
$\epsilon_1,...,\epsilon_n$ standard Gaussian random variables, define $X_f := \sum_{i=1}^n \epsilon_i f(x_i) / \sqrt{n}$. Take any two functions $f, f'\in\mathcal{F}$ and note that $X_f - X_{f'}$ follows a Gaussian distribution with expected value $\mathbb{E}[X_f - X_{f'}]=0$ and variance

\begin{equation*}
Var(X_f-X_{f'}) = ||f||^2_n + ||f'||^2_n - 2Cov(X_f,X_{f'}) = ||f - f' ||^2_n.
\end{equation*}

Therefore $\{ X_f : f\in\mathcal{F}(R)\}$ is a sub-Gaussian process with respect to the metric $d(f,f')=||f -f'||_n$ on its index set (see Appendix). Note that, if we define the diameter of $\mathcal{F}(R)$ as $diam_n(\mathcal{F}(R)) := \sup_{x,y \in\mathcal{F}(R)} || x-y ||_n$, then it is not difficult to see that $ diam_n(\mathcal{F}(R))\leq 2R$.\\
~\\
Now we proceed to obtain bounds for $M(R)$. By Dudley's entropy bound (see Lemma \ref{Koltch2001} in Appendix) and Condition \ref{Ss:CoveringNumberCondition}, for some constants $C\geq 1/2$ and $c_2$, we have 

\begin{align*}
M(R) = \frac{1}{\sqrt{n}} \mathbb{E} \left[ \sup_{f\in\mathcal{F}(R)} X_f - X_{f^0} \right] 
&  \leq C \int_{0}^{2R} \sqrt{\frac{\log N(u,\mathcal{F}(R) ,
  ||\cdot||_n)}{n}} \mathrm{d}u \\
& \leq \frac{C \sqrt{c_2}}{\sqrt{n}} \int_0^{2R} \left( \frac{R}{u
  \lambda} \right)^{\alpha/2} \mathrm{d}u \\
 & \leq \frac{4 C \sqrt{c_2}}{2-\alpha} \left( \frac{1}{n
  \lambda^\alpha}\right)^{1/2} R. \\
\end{align*}

Moreover, by Sudakov lower bound (see Lemma \ref{SudakovIneq} in Appendix) we have

\begin{equation*}
\frac{1}{2} \sup_{0 <\epsilon \leq diam_n(\mathcal{F}(R))} \epsilon \sqrt{\frac{\log N_D (\epsilon, \mathcal{F}(R), || \cdot ||_n )}{n}} \leq M(R).
\end{equation*}

For some $0 < c_0 \leq 2 $, let $diam_n(\mathcal{F}(R))= c_0 R$ and take $\epsilon = c_0 R$ in the last equation. By Condition \ref{Ss:CoveringNumberCondition}, for some constant $c_1$ we have 

\begin{align*}
\frac{c_0 R}{2} \sqrt{\frac{\log N_D ( c_0 R , \mathcal{F}(R), || \cdot ||_n )}{n}}
& \geq \frac{c_0 R}{2} \sqrt{\frac{\log N( c_0 R , \mathcal{F}(R), || \cdot ||_n )}{n}}\\
& \geq \frac{c_0^{1 - \alpha/2}\sqrt{c_1}}{2} \left( \frac{1}{n \lambda^\alpha}  \right)^{1/2} R. 
\end{align*}

Then, by the equations above and the definition of $H(R)$, we have, for all $R \geq R_{min}$, 

\begin{equation*}
\frac{c_0^{1 - \alpha/2}\sqrt{c_1}}{2} \left( \frac{1}{n \lambda^\alpha}
\right)^{1/2} R - \frac{R^2}{2} < H(R) < \frac{4 C \sqrt{c_2}}{2-\alpha} \left( \frac{1}{n
  \lambda^\alpha}\right)^{1/2} R - \frac{R^2}{2}.
\end{equation*}

Writing $K_iR-R^2/2 = -\frac{1}{2}(R-K_i)^2 +
K_i^2/2$  for $i=1,2$ completes the proof.
\end{proof}

Now, we are ready to prove Theorem \ref{Theorem:BoundsTau}.

\begin{proof} [Proof of Theorem \ref{Theorem:BoundsTau}] By Condition \ref{Ss:RminCondition} and $\mathcal{I}(f^0) \asymp 1$, we know that there exist constants $0 < b_1 < b_2$ not depending on $n$ such that

\begin{equation}\label{proof:SimpleCondRmin}
b_1 \lambda \leq R_{min} \leq b_2 \lambda .
\end{equation}

Take 
\begin{equation}\label{proof:ChoiceLambda}
c'n^{-\frac{1}{2+\alpha}} \leq \lambda \leq c'' n^{-\frac{1}{2+\alpha}}
\end{equation}

with $c',c''$ constants satisfying $0 < c' < c'' \leq \left(\frac{c_0^{1 - \alpha/2}\sqrt{c_1}}{2 b_2}\right)^{\frac{2}{2+ \alpha}}$, where $\alpha$ is as in Condition \ref{Ss:CoveringNumberCondition}, $c_0$ and $c_1$ are as in Lemma \ref{Lemma_BoundsHR}, and $b_2$ as in equation \eqref{proof:SimpleCondRmin}.\\
~\\
First, we will derive bounds for $R_0$. 
Let $g_1$ and $g_2$ be as in Lemma \ref{Lemma_BoundsHR} and recall that $g_1(R) < H(R) < g_2(R)$ for $R \geq R_{min}$. Note that $g_1(R)\leq K_1^2/2$ for all $R \geq R_{min}$ and that this upper bound is reached at $R=K_1$. Moreover, we know that the function $g_2$ attain negative values when $R<0$ and when $R > 2K_2$.  Then, by strict concavity of $H$ (Lemma \ref{Lemma:H_concave}), if $R_{min}\leq K_1$, we must have that $R_0 \leq 2K_2$.\\
~\\
Now, combining equation \eqref{proof:SimpleCondRmin} and the choice in \eqref{proof:ChoiceLambda}, we obtain that $R_{min}\leq K_1$. Therefore, following the rationale from above and substituting equation $\eqref{proof:ChoiceLambda}$ into the definition of $K_2$, we have that there exist some constant $a_2 >0$ such that

\begin{equation}\label{eq:UboundsRE}
R_0 \leq a_2 n^{-\frac{1}{2 + \alpha}} 
\end{equation}

Furthermore, combining again equations \eqref{proof:SimpleCondRmin} and \eqref{proof:ChoiceLambda}, and recalling that $R_{min}\leq R_0$ yields that, for some constant $a_1 >0$, 

\begin{equation}\label{eq:LboundsRE}
R_0 \geq a_1 n^{-\frac{1}{2 + \alpha}} 
\end{equation}

Joining equations \eqref{eq:UboundsRE} and \eqref{eq:LboundsRE} gives us the first result in our theorem. \\
~\\
We proceed to obtain bounds for $\tau(\hat{f})$. Let $A_1,A_2$ be some constants such that\\ $0 < A_1 < a_1 < a_2 < A_2$. We have

\begin{align*}
P \left(  \tau(\hat{f}) \leq A_1 n^{-\frac{1}{2+\alpha}}  \right) 
& = P \left(  R_0 - \tau(\hat{f}) \geq R_0 - A_1 n^{-\frac{1}{2+\alpha}}  \right)  \\
& \leq P \left(  R_0 - \tau(\hat{f}) \geq  (a_1- A_1) n^{-\frac{1}{2+\alpha}}  \right) \\
& \leq 3 \exp \left( {-\frac{(a_1 - A_1)^4 \, n^{\frac{\alpha}{2 + \alpha}} }{32(a_2 + a_1 - A_1)}} \right),
\end{align*}

where in the first inequality we used equation \eqref{eq:LboundsRE}, and in the second, Theorem \ref{Theorem_CI} and equation \eqref{eq:UboundsRE}. Similarly, we have 

\begin{align*}
P \left(  \tau(\hat{f}) \geq A_2 n^{-\frac{1}{2+\alpha}} \right) 
& \leq P \left(  \tau(\hat{f}) - R_0 \geq (A_2-a_2) n^{-\frac{1}{2+\alpha}}  \right)  \\
&    \leq 3 \exp \left( {-\frac{(A_2 - a_2)^4 \, n^{\frac{\alpha}{2 + \alpha}} }{32 A_2}} \right),\\
\end{align*}

where in the first inequality we use equation \eqref{eq:UboundsRE}, and in the second, Theorem \ref{Theorem_CI} and again equation \eqref{eq:UboundsRE}. Therefore, we have 

\begin{align*}
~
& P \left( A_1 n^{-\frac{1}{2+\alpha}} \leq  \tau(\hat{f}) \leq A_2 n^{-\frac{1}{2+\alpha}}  \right)\\
&  =P \left(  \tau(\hat{f}) \leq A_2 n^{-\frac{1}{2+\alpha}}  \right) - P \left(  \tau(\hat{f}) \leq A_1 n^{-\frac{1}{2+\alpha}}  \right) \\
&  \geq 1 - 3   \exp \left( {-\frac{(a_1 - A_1)^4 \, n^{\frac{\alpha}{2 + \alpha}} }{32(a_2 + a_1 - A_1)}} \right)-3\exp \left( {-\frac{(A_2 - a_2)^4 \, n^{\frac{\alpha}{2 + \alpha}} }{32 A_2}} \right) 
\end{align*}

and the second result of the theorem follows.
\end{proof}


\printbibliography

\appendix
\section{Appendix}

\begin{Lemma}\label{GaussianConcIneq}[Gaussian Concentration Inequality. See, e.g. \textcite{ConcIneqBouLuMa13}]

Let $X=(X_1,...,X_n)$ be a vector of $n$ independent standard Gaussian random
variables. Let $f:\mathbb{R}^n \longrightarrow \mathbb{R}$ denote an
L-Lipschitz function. Then, for all $t>0$,

\begin{equation*}
\mathbb{P} \left( f(X) - \mathbb{E}f(X) \geq t  \right) \leq e^{-t^2/(2L^2)}.
\end{equation*}
\end{Lemma}
~\\
The following lemma applies the Gaussian Concentration Inequality from above to show that the quantities $H_n$ and $H$ are close with exponential probability, as exploited by \textcite{ChatLSE2014}.
~\\
\begin{Lemma}\label{ConcIneqG}
For all $\lambda>0$, $R \geq R_{min}$,  and $t >0$, we have
\begin{equation*}
\mathbb{P} \left( |H_n(R) - H(R)| \geq t
\right) \leq 2e^{- n\,t^2/(2R^2)} .
\end{equation*}
\end{Lemma}

\begin{proof}
Let $\lambda>0$. In this proof, we will write

\begin{equation*}
M_n(R) = M_n(R,\epsilon) = \sup_{f\in\mathcal{F}(R)}  \langle \epsilon, f- f^0 \rangle/n. 
\end{equation*}

Let $u$ and $v$ be two $n$-dimensional standard Gaussian random
vectors. By properties of the supremum and by Cauchy-Schwarz inequality, we have 

\begin{align*}
| M_n(R,u) - M_n(R,v) |
& = \left\vert \sup_{f\in\mathcal{F}(R)} \langle u, f-f^0\rangle/n -
  \sup_{f\in\mathcal{F}(R)} \langle v, f-f^0\rangle/n  \right\vert  \\
& \leq \sup_{f\in\mathcal{F}(R)} \Bigg\vert \langle u, f-f^0\rangle/n -
  \langle v, f-f^0\rangle/n
\Bigg\vert \\
& \leq \sup_{f\in\mathcal{F}(R)}  ||u-v||_n \;||f-f^0||_n \leq \frac{R}{\sqrt{n}} \left(\sum_{i=1}^n (u_i-v_i)^2\right)^{1/2}.
\end{align*}

Therefore, $M_n(R,\cdot)$ is $(R/\sqrt{n})$ - Lipschitz in its
second argument. By the Gaussian Concentration Inequality (Lemma \ref{GaussianConcIneq}), for every $t>0$ and every $R\geq R_{min}$, we have
\begin{equation*}
\mathbb{P} \left( M_n(R,\epsilon) - M(R,\epsilon) \geq t
\right) \leq e^{- n\,t^2/(2R^2)}.
\end{equation*}

Now, take $-M_n(R,\epsilon)$. Applying again the Gaussian Concentration Inequality yields

\begin{equation*}
\mathbb{P} \left( M_n(R,\epsilon) - M(R,\epsilon) \leq -t
\right) \leq e^{- n\,t^2/(2R^2)}.
\end{equation*}

Combining the last two equations yields the result of this lemma.
\end{proof}
~\\
For the next lemma, we will need the following definition:\\
~\\
A stochastic process $\{ X_t:t\in T \}$ is called \textit{sub-Gaussian with respect to the semi-metric
$d$ on its index set} if 
\begin{equation*}
P (|X_s - X_t|>x) \leq 2 e^{-\frac{x^2}{2 d^2(s,t)}},\quad\quad
\text{for every}\; s,t\in T,\,x>0.
\end{equation*}

\begin{Lemma}\label{Koltch2001}[Dudley's entropy bound. See, e.g. \textcite{koltchinskii2011oracle} ] If $\{ X_t : t\in T \}$ is a sub-Gaussian process with respect to $d$, then the following bounds hold with some numerical constant $C>0$:

\begin{equation*}
\mathbb{E} \sup_{t\in T} X_t \leq C \int_0^{D(T)} \sqrt{\log N(\epsilon, T, d)} \,\mathrm{d}\epsilon
\end{equation*}

and for all $t_0\in T$
\begin{equation*}
\mathbb{E} \sup_{t\in T} |X_t - X_{t_0}| \leq C \int_0^{D(T)} \sqrt{\log N(\epsilon, T, d)}\, \mathrm{d}\epsilon
\end{equation*}

where $D(T)=D(T,d)$ denotes the diameter of the space $T$.

\end{Lemma}


\begin{Lemma}\label{SudakovIneq}[Sudakov Lower Bound. See, e.g.  \textcite{ConcIneqBouLuMa13} ] Let $T$ be a finite set and let $(X_t)_{t\in T}$ be a
Gaussian vector with $\mathbb{E}X_t = 0\,\, \forall t$. Then,

\begin{equation*}
\mathbb{E} \sup_{t\in T} X_t \geq \frac{1}{2} \min_{t \neq t'}
\sqrt{\mathbb{E} [(X_t-X_{t'})^2] \log | T|}.
\end{equation*}

Moreover, let $d$ be a pseudo-metric on $ T$ defined by $d(t,t')^2 = \mathbb{E} [(X_t-X_{t'})^2]$. For all\\ $\epsilon>0$ smaller than the diameter of
$T$, the lower bound from above can be rewritten as 

\begin{equation*}
\mathbb{E} \sup_{t\in T} X_t \geq \frac{1}{2} \epsilon
\sqrt{\log N_D(\epsilon, T,d)}.
\end{equation*}
\end{Lemma}

\end{document}